\documentclass[graybox]{SNmult}

\usepackage{type1cm}        
\usepackage{makeidx}         
\usepackage{graphicx}        
\usepackage{multicol}        
\usepackage[bottom]{footmisc}

\usepackage{newtxtext}       %
\usepackage[varvw]{newtxmath}       

\usepackage{mathrsfs,mathtools}
\usepackage{tikz}
\usetikzlibrary{patterns}
\usepackage{psfrag}
\usepackage{pgfplots}
\usepackage{color}

\makeindex             

\let\epsilon\varepsilon
\newcommand{\tH}{\widetilde{H}}

\newcommand{\hPi}{\widehat{\Pi}}

\newcommand{\N}{\mathbb{N}}
\newcommand{\R}{\mathbb{R}}

\newcommand{\lay}{\mathcal{L}^L}

\newcommand{\cutoff}{g^L}

\newcommand{\hpartial}{\widehat{\partial}}
\newcommand{\hu}{\widehat{u}}

\newcommand{\Tg}{\mathcal{T}^{L}_{\mathrm{geo},\sigma}}

\newcommand{\xh}{{\widehat x}}
\newcommand{\yh}{{\widehat y}}
\newcommand{\zh}{{\widehat z}}

\newcommand{\Kh}{{\widehat  K}}

\newcommand{\skp}[1]{\left< #1 \right>}
\newcommand{\norm}[1]{\left\| #1 \right\|}

\newcommand{\abs}[1]{\lvert#1\rvert}

\newcommand{\bbQ}{{\mathbb Q}}
\newcommand{\bbP}{{\mathbb P}}
\newcommand{\bbR}{{\mathbb R}}

\newcommand{\bfe}{\mathbf{e}}
\newcommand{\bfv}{\mathbf{v}}
\newcommand{\bff}{\mathbf{f}}

\newcommand{\cV}{\mathcal{V}}
\newcommand{\cF}{\mathcal{F}}
\newcommand{\cE}{\mathcal{E}}

\newcommand{\VLq}{W^L_q}
\newcommand{\PiLq}{\Pi^L_q}

\newcommand{\omegaeps}{\xi}

\newcommand{\omegav}{\omega_{\mathbf{v}}}

\newcommand{\omegave}{\omega_{\mathbf{ve}}}
\newcommand{\omegavf}{\omega_{\mathbf{vf}}}
\newcommand{\omegavef}{\omega_{\mathbf{vef}}}
\newcommand{\omegaef}{\omega_{\mathbf{ef}}}

\newcommand{\omegae}{\omega_{\mathbf{e}}}
\newcommand{\omegaf}{\omega_{\mathbf{f}}}

\newcommand{\betam}{|\beta|}

\newcommand{\parperp}{\vDash}

\newcommand{\gpar}{\mathbf{g}_{\parallel}}
\newcommand{\gparperp}{\mathbf{g}_{\parperp}}
\newcommand{\gperp}{\mathbf{g}_{\perp}}

\newcommand{\betaperp}{\beta_{\perp}}

\newcommand{\betapar}{\beta_{\parallel}}

\newcommand{\betaparperp}{\beta_{\parperp}}

\newcommand{\rv}{r_\mathbf{v}}
\newcommand{\re}{r_\mathbf{e}}
\newcommand{\rf}{r_\mathbf{f}}
\newcommand{\rhove}{\rho_\mathbf{ve}}
\newcommand{\rhoef}{\rho_\mathbf{ef}}
\begin{document}
\title*{Exponential Convergence of $hp$-FEM \\
for the Integral Fractional Laplacian on cuboids}
\titlerunning{Exponential Convergence for Integral Fractional Laplacian}
\author{
Bj\"orn Bahr \and
Markus Faustmann \and 
Carlo Marcati \and \\
Jens Markus Melenk  \and
Christoph Schwab}

\institute{
B. Bahr, M. Faustmann, J.M.Melenk \at Institut  f\"{u}r  Analysis und  Scientific Computing\\ 
Technische Universit\"{a}t Wien\\ A-1040  Vienna, Austria \\ \email{bjoern.bahr@tuwien.ac.at, markus.faustmann@tuwien.ac.at, melenk@tuwien.ac.at}
\and
C. Marcati \at Institut Camille Jordan\\
Université de Lyon, Université Claude Bernard Lyon 1, CNRS UMR 5208\\
F-69622 Villeurbanne, France \\
\email{carlo.marcati@univ-lyon1.fr}
\and
Ch. Schwab \at Seminar for Applied Mathematics\\ ETH Z\"{u}rich, ETH Zentrum, 
HG  G57.1\\ CH8092 Z\"{u}rich, Switzerland \\ \email{christoph.schwab@sam.math.ethz.ch}
}

\maketitle
\abstract*{For the Dirichlet integral fractional Laplacian, 
we prove root exponential convergence of tensor-product $hp$-finite element approximations 
on $(0,1)^3$, 
for forcing $f$ that is analytic in $[0,1]^3$.
Exploiting analytic regularity estimates in weighted Sobolev spaces \cite{FMMS25-regularity3d},
we prove for $hp$-GLL interpolation approximations                      
with $N$ degrees of freedom the energy norm error bound $\lesssim \exp(-b\sqrt[6]{N})$. 
Tensor product mesh families which are geometrically refined towards all sides of $(0,1)^3$ 
are used. Numerical experiments with $hp$-Galerkin FEM confirm the bound.
}

\abstract{For the Dirichlet integral fractional Laplacian, 
we prove root exponential convergence of tensor-product $hp$-finite element approximations 
on $(0,1)^3$, 
for forcing $f$ that is analytic in $[0,1]^3$.
Exploiting analytic regularity estimates in weighted Sobolev spaces \cite{FMMS25-regularity3d},
we prove for $hp$-GLL interpolation approximations                      
with $N$ degrees of freedom the energy norm error bound $\lesssim \exp(-b\sqrt[6]{N})$. 
Tensor product mesh families which are geometrically refined towards all sides of $(0,1)^3$ 
are used. Numerical experiments with $hp$-Galerkin FEM confirm the bound.
}

\section{Introduction and Main Result}
\label{S:introduction}
%

Fractional differential operators and their 
applications in non-local physics, finance or biology attracted 
substantial attention in recent years \cite{BV16,LPGSGZMCMAK18}.

In order to derive meaningful simulations of the underlying fractional differential equations, 
such as the fractional Poisson problem $(-\Delta)^s u = f$, $0<s<1$,
the non-local and singular nature of these problems requires the 
use of efficient and accurate numerical methods. 
Among the various choices considered in literature, \cite{BBNOS18,GunbActa}, 
we consider the $hp$-version of the Finite Element Method (``$hp$-FEM''), 
which combines anisotropic, geometrically boundary refined partitions 
with higher order polynomials. 
Exponential convergence of $hp$-FEM applied to the 
fractional Poisson problem has been observed in our previous works 
in space dimensions $d=1$, $2$, \cite{BFMMS-hp1d,FMMS23-hp2d}. 
The aim of this article is to extend this analysis to the more challenging case 
of $d=3$ spatial dimensions; 
we restrict the computational domains to cuboids in order to simplify the presented proofs.
\subsection{Integral Fractional Diffusion}
\label{sec:FracDiff}

Let $\Omega := (0,1)^3 \subset \bbR^3$ be the unit cube. 
Given $f\in L^2(\Omega)$,
we consider fractional differential equations of the form:
Find $u:\bbR^3\to \bbR$ such that
\begin{equation}\label{eq:FracLap}
(-\Delta)^s u = f \quad \mbox{in}\quad \Omega, 
\qquad 
u = 0 \quad \mbox{in} \;\;\Omega^c := \bbR^3\backslash \overline{\Omega}
.
\end{equation}
Here, the Dirichlet integral fractional Laplacian $(-\Delta)^s$ is, for $0<s<1$, 
given by
\begin{equation}\label{eq:IntFrc}
(-\Delta)^s u(x) := C(s) \mbox{P.V.} \int_{\bbR^3} \frac{u(x)-u(z)}{|x-z|^{3+2s}} dz 
\;,
\quad 
C(s)\coloneqq - 2^{2s}\frac{\Gamma(s+3/2)}{\pi^{3/2}\Gamma(-s)},
\end{equation}
where $\mbox{P.V.}$ denotes the Cauchy principal value integral and $\Gamma(\cdot)$ is the gamma-function.
Suitable function spaces for fractional differential equations are fractional Sobolev spaces, 
defined by means of the Sobolev-Slobodeckij seminorm. For $t>0$ and $\omega \subset \bbR^3$, 
this reads
\begin{align}\label{eq:FracNrm}
|v|^2_{H^t(\omega)}
=
\int_{\omega} \int_{\omega} \frac{|v(x) - v(z)|^2}{\abs{x-z}^{3+2t}}
\,dz\,dx,
\;\;
\|v\|^2_{H^t(\omega)} = \|v\|^2_{L^2(\omega)} + |v|^2_{H^t(\omega)}.
\end{align}
In order to incorporate the exterior Dirichlet boundary condition, 
for $t \in (0,1)$, we employ the spaces
\begin{align} \label{eq:Htilde}
\widetilde{H}^{t}(\Omega) 
:= 
\left\{u \in H^t(\R^3) : u\equiv 0 \; \text{on} \; \R^3 \backslash \overline{\Omega} \right\},
\;  
\norm{v}_{\widetilde{H}^{t}(\Omega)}^2 
:= 
\norm{v}_{H^t(\Omega)}^2 + \norm{v/r_{\partial\Omega}^t}_{L^2(\Omega)}^2.
\end{align}
Here and throughout, $r_{\partial \Omega}(x):=\operatorname{dist}(x,\partial\Omega)$
denotes the Euclidean distance of a point $x \in \Omega$ from the boundary $\partial \Omega$.
For $t > 0$, the space $H^{-t}(\Omega)$ denotes the dual space of $\widetilde{H}^t(\Omega)$,
and $\skp{\cdot,\cdot}_{L^2(\Omega)}$ denotes 
the duality pairing that extends the $L^2(\Omega)$-inner product.

The variational form of \eqref{eq:FracLap} 
reads: find $u \in \widetilde{H}^s(\Omega)$ such that, 
for all $v \in \widetilde{H}^s(\Omega)$,
\begin{equation}
\label{eq:weakform}
a(u,v):= \frac{C(s)}{2} \int_{\R^3}\int_{\R^3}
 \frac{(u(x)-u(z))(v(x)-v(z))}{\abs{x-z}^{3+2s}} \, dz \, dx = \skp{f,v}_{L^2(\Omega)}.
\end{equation}
Existence and uniqueness of $u \in \widetilde{H}^s(\Omega)$ follow from
the Lax--Milgram Lemma for any $f \in H^{-s}(\Omega)$, upon the observation
that the bilinear form 
$a(\cdot,\cdot): \widetilde{H}^s(\Omega)\times \widetilde{H}^s(\Omega)\to \R$
is continuous and coercive,
see, e.g., \cite[Sec.~{2.1}]{acosta-borthagaray17}.

\subsection{$hp$-FEM discretization on geometrically refined meshes}
\label{sec:mesh}

As we are in the setting of the Lax-Milgram lemma, we have, for each closed subspace 
$V_N\subset \widetilde{H}^s(\Omega)$, 
unique solvability of the variational formulation: Find 
\begin{align}\label{eq:GalV}
u_N\in V_N \; \mbox{ such that }\;
a(u_N,v) = \skp{f,v}_{L^2(\Omega)} \quad \forall v\in V_N.
\end{align}
The C\'ea-Lemma gives quasi-optimality of the Galerkin approximation
\begin{align}\label{eq:QuasiOpt}
\forall v_N \in V_N:\quad 
\| u - u_N \|_{\widetilde{H}^s(\Omega)} 
\leq C 
\| u - v_N \|_{\widetilde{H}^s(\Omega)} 
\;.
\end{align}

The aim of this work is to construct a FEM space $V_N$ such that the best-approximation error converges (root) exponentially in the dimension of the space. 
In particular, we take a space of piecewise polynomials on certain geometrically refined meshes constructed in the following. 

Given a geometric mesh grading parameter $\sigma \in (0,1)$, 
we consider tensor product meshes that are geometrically refined towards all faces 
(see Figure~\ref{fig:mesh_refinement}). 
These meshes are constructed by tensorization of geometric partitions in $(0,1)$
with nodes
$$
x_0 = 0, \; x_i = \frac{\sigma^{L-i}}{2}, i = 1,\dots,L,\quad x_{i} = 1-\frac{\sigma^{i-L}}{2}, i=L+1,\dots,2L, \; x_{2L+1} = 1.
$$ 
Here, $L \geq 1$ denotes the number of layers of refinement towards the boundary points.

\begin{figure}[ht]
\begin{center}
\begin{tikzpicture}[scale=0.35]
\draw[ultra thick] (0,0)--(8,0); 
\draw[ultra thick] (0,8)--(8,8);
\draw[ultra thick] (8,0) -- (8,8); 
\draw[ultra thick] (0,0) -- (0,8); 

\draw[ultra thick] (0,8) -- (4,12); 
\draw[ultra thick] (8,8) -- (12,12); 
\draw[ultra thick] (8,0) -- (12,4); 
\draw[ultra thick] (12,4) -- (12,12); 
\draw[ultra thick] (4,12) -- (12,12); 
\draw (0.25,0) -- (0.25,8); 
\draw (0.5,0) -- (0.5,8); 
\draw (1,0) -- (1,8); 
\draw (2,0) -- (2,8); 
\draw (4,0) -- (4,8); 
\draw (6,0) -- (6,8); 
\draw (7,0) -- (7,8); 
\draw (7.5,0) -- (7.5,8); 
\draw (7.75,0) -- (7.75,8); 
\draw (0,0.25) -- (8,0.25);
\draw (0,0.5) -- (8,0.5); 
\draw (0,1) -- (8,1); 
\draw (0,2) -- (8,2); 
\draw (0,4) -- (8,4); 
\draw (0,6) -- (8,6); 
\draw (0,7) -- (8,7); 
\draw (0,7.5) -- (8,7.5); 
\draw (0,7.75) -- (8,7.75);  
\draw (8.125,0.125) -- (8.125,8.125); 
\draw (8.25,0.25) -- (8.25,8.25); 
\draw (8.5,0.5) -- (8.5,8.5); 
\draw (9,1) -- (9,9); 
\draw (10,2) -- (10,10); 
\draw (11,3) -- (11,11); 
\draw (11.5,3.5) -- (11.5,11.5); 
\draw (11.75,3.75) -- (11.75,11.75); 
\draw (11.875,3.875) -- (11.875,11.875); 
\draw (8,0.25) -- (12,4.25);
\draw (8,0.5) -- (12,4.5); 
\draw (8,1) -- (12,5); 
\draw (8,2) -- (12,6); 
\draw (8,4) -- (12,8); 
\draw (8,6) -- (12,10); 
\draw (8,7) -- (12,11); 
\draw (8,7.5) -- (12,11.5); 
\draw (8,7.75) -- (12,11.75); 
\draw (0.25,8) -- (4.25,12); 
\draw (0.5,8) -- (4.5,12); 
\draw (1,8) -- (5,12); 
\draw (2,8) -- (6,12); 
\draw (4,8) -- (8,12); 
\draw (6,8) -- (10,12); 
\draw (7,8) -- (11,12); 
\draw (7.5,8) -- (11.5,12); 
\draw (7.75,8) -- (11.75,12); 
\draw (0.125,8.125) -- (8.125,8.125);
\draw (0.25,8.25) -- (8.25,8.25); 
\draw (0.5,8.5) -- (8.5,8.5); 
\draw (1,9) -- (9,9); 
\draw (2,10) -- (10,10); 
\draw (3,11) -- (11,11); 
\draw (3.5,11.5) -- (11.5,11.5); 
\draw (3.75,11.75) -- (11.75,11.75); 
\draw (3.875,11.875) -- (11.875,11.875); 
\end{tikzpicture}
\caption{\label{fig:mesh_refinement} Geometric mesh $\Tg$ with $\sigma=1/2$ and $L=4$.}
\end{center}
\end{figure}
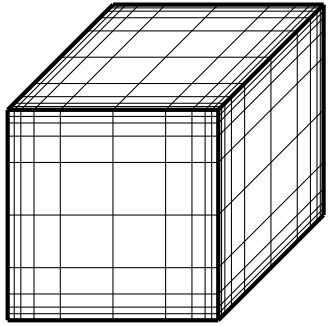

Let $\widehat K:= (0,1)^3$ be the reference cube. 
For each (axisparallel) hexahedron $K \in \Tg$, 
the element transformation $F_K : \widehat K \rightarrow K$ 
is affine, with diagonal Jacobian.

For a function $u:\Omega \rightarrow \R$ 
and for $K\in \Tg$, 
we denote by  ${\widehat u}_K : \widehat K \rightarrow \R$
the pull-back of $u$ via $F_K$ to the reference element, 
i.e.,
\begin{equation}\label{eq:hatu}
{\widehat u}_K := u|_K \circ F_K.
\end{equation}

On the geometric mesh $\Tg$,
we consider Lagrangian finite elements of uniform polynomial degree $q\geq 1$, 
i.e. the spaces $\VLq\coloneqq \mathcal{S}^q_0(\Omega,\Tg)$, where
\begin{equation}
\label{eq:S^q_0}
\mathcal{S}^q_0(\Omega,\Tg)
\coloneqq 
\left \{ v \in C(\overline{\Omega}): {\widehat v}_K \in \mathbb{Q}_{q}(\widehat K) 
\quad 
\forall K \in \Tg, \ v|_{\partial \Omega} = 0 \right \}.
\end{equation}
Here, $\mathbb{Q}_{q}(\widehat K) := (\mathbb{P}_q)^{\otimes 3}$
and 
$\operatorname*{dim}W^L_q = O(q^3 L^3)$.
\subsection{Main Result}
\label{S:contrib}
%
\begin{theorem}\label{thm:hpExpConv}
Let $\Omega := (0,1)^3$, $f$ be analytic in  $\overline{\Omega}$ and $u$ solve \eqref{eq:weakform}. 
Then, 
the Galerkin approximations $u_N \in V_N := W_q^L$ of \eqref{eq:GalV}  with $W^L_q$ defined in \eqref{eq:S^q_0} with $L\sim q \sim N^{1/6}$
converge exponentially to $u$, i.e., there 
are constants $b$, $C >0$ (depending on $\Omega$, $f$, and $s$) 
such that
\begin{equation}\label{eq:hpExpConv}
\| u - u_N \|_{\tH^s(\Omega)} \leq C \exp(-b \sqrt[6]{N}).
\end{equation}
%
%
\end{theorem}

The remainder of this note is a 
proof of Theorem~\ref{thm:hpExpConv}, which will be given in Section~\ref{sec:proof}. 
Its main ingredient are suitable regularity estimates provided by our previous work, \cite{FMMS25-regularity3d}, which are summarized in Section~\ref{sec:regularity}.
Finally, a numerical example validates the theoretical statement of our main result in Section~\ref{sec:numerics}.
\bigskip

\textbf{Novelty:} The present article is the first contribution in literature that provides a proven (root) exponential error bound for approximations to the Dirichlet integral fractional Laplacian in three dimensions with analytic forcing. Moreover, the numerical example in three dimensions is to our knowledge the first implementation in three dimensions that realizes an exponential convergent $hp$-FEM.
\section{Weighted Analytic Regularity}
\label{sec:regularity}

The key to deriving the exponential convergence rate for $hp$-FEM discretization are regularity results for the exact weak solution. Here, we consider the case of analytic regularity in weighted Sobolev spaces as shown in our previous work \cite{FMMS25-regularity3d}. The idea therein is to mitigate the regularity loss close to $\partial \Omega$ by suitable multiplication with powers of the distance function $r_{\partial\Omega}(\cdot)$. 

In this section, we briefly state the regularity results from \cite{FMMS25-regularity3d}. An important tool hereby is to make use of a decomposition of the domain w.r.t. the geometric situation of singular sets (vertices, edges, faces) meeting each other. 

\subsection{Partition of $\Omega$}
\label{sec:PartOmega}
For each vertex $\bfv \in \cV$, 
we denote by 
$\cE_{\bfv}\coloneqq \{\bfe \in \cE: \bfv \in \overline{\bfe}\}$
the set of all edges that meet at $\bfv$ and by 
$\cF_{\bfv}\coloneqq \{\bff \in \cF: 
 \overline{\bff} \cap \bfv \ne \emptyset\}$
the set of all faces abutting the vertex $\bfv$.
For any edge $\bfe \in \cE$, 
we denote by 
$\cV_{\bfe}\coloneqq \{\bfv \in \cV:  \bfv \in \overline{\bfe}\}  =  \partial {\bfe}$ 
the endpoints of edge $\bfe$ and by 
$\cF_{\bfe}\coloneqq \{\bff \in \cF: \overline{\bff} \cap \bfe \ne \emptyset\}$ 
the set of faces sharing the edge $\bfe$.
For any face $\bff\in\cF$,
$\cE_{\bff} \coloneqq \{\bfe \in \cE\,:\, \bfe \subset \partial{\bff} \}$ 
is the set of edges abutting the face $\bff$, and 
$\cV_{\bff} \coloneqq \{\bfv \in \cV\,:\, \bfv \in \overline{\bff} \}$ is
the set of vertices contained in the face $\overline{\bff}$.
\bigskip

For 
$\bfv \in \cV$, 
$\bfe \in \cE$, and
$\bff \in \cF$,
we shall require the distance functions 
\begin{align*} 
  \rv(x)\coloneqq|x - \bfv|, 
  \qquad 
  \re(x)\coloneqq\inf_{y \in \bfe} |x - y|, 
  \qquad
  \rf(x)\coloneqq\inf_{y \in \bff} |x - y|,
  \quad 
  x\in \Omega,
\end{align*}
and corresponding (nondimensional) relative distances
\begin{align*}
  \rhove(x)\coloneqq \re(x)/\rv(x),
  \qquad 
  \rhoef(x)\coloneqq \rf(x)/\re(x).
\end{align*} 

We start with a decomposition of $\Omega$ into (possibly overlapping) sectorial neighborhoods of vertices $\mathbf{v}$, edges $\mathbf{e}$
and faces $\mathbf{f}$ and an interior $\Omega_{\rm int}$. Each sectorial neighborhood my be further divided, if a second or third type of singular set is contained in the neighboorhood. 
More precisely, we write
\begin{align}\label{eq:Nghbrhoods}
 \Omega = \Omega_{\rm int} \cup \bigcup_{\mathbf{v} \in \mathcal{V}}\left( \omega_{\mathbf{v}} 
           \cup \bigcup_{\mathbf{e} \in \mathcal{E}_{\mathbf{v}}, \; \mathbf{f} \in \mathcal{F}_{\mathbf{v}}} 
           \omega_{\mathbf{ve}} \cup \omega_{\mathbf{vf}} \cup \omegavef \right) 
           \cup \bigcup_{\mathbf{e} \in \mathcal{E}} \left( \omega_{\mathbf{e}} 
           \cup \bigcup_{\mathbf{f} \in \mathcal{F}_{\mathbf{e}}} \omega_{\mathbf{ef}} \right)
           \cup \bigcup_{\mathbf{f} \in \mathcal{F}} \omega_{\mathbf{f}}
\;.
\end{align}
Here, for fixed, sufficiently small $\omegaeps > 0$ and for 
$\bfv\in\cV$, $\bfe\in\cE$, $\bff\in\cF$,
the various neighborhoods are given by
  \begin{align*}
    \omegavef^{\omegaeps} & \coloneqq \{x \in \Omega\,:\, \rv(x) < \omegaeps \quad \wedge \quad \rhove(x) < \omegaeps \quad \wedge \quad \rhoef(x) < \omegaeps \},
    \\
    \omegave^{\omegaeps} &\coloneqq \{x \in \Omega\,:\, \rv(x) < \omegaeps \quad \wedge \quad \rhove(x) < \omegaeps \quad \wedge \quad \rhoef(x) \geq \omegaeps \quad \forall \bff \in \mathcal{F}_{\bfe} 
    \},
    \\
    \omegavf^{\omegaeps} &\coloneqq \{x \in \Omega\,:\, \rv(x) < \omegaeps \quad \wedge \quad \rhove(x) \geq \omegaeps \quad \wedge \quad \rhoef(x) < \omegaeps \quad \forall \bfe \in \mathcal{E}_{\bfv} \cap \mathcal{E}_{\bff} \},
    \\
   \omegav^{\omegaeps} & \coloneqq \{x \in \Omega\,:\, \rv(x) < \omegaeps \quad \wedge \quad \rhove(x) \geq \omegaeps \quad \wedge \quad \rhoef(x) \geq \omegaeps \quad \forall \bfe \in \mathcal{E}_{\bfv},\; \bff \in \mathcal{F}_{\bfv} \}, 
    \\
    \omegaef^{\omegaeps} &\coloneqq \{x \in \Omega\,:\, \rv(x) \geq \omegaeps \quad \wedge \quad \re(x) < \omegaeps^2 \quad \wedge \quad \rhoef(x) < \omegaeps \quad \forall \bfv \in \mathcal{V}_{\bfe} \},
    \\
    \omegae^{\omegaeps} &\coloneqq \{x \in \Omega\,:\, \rv(x) \geq \omegaeps \quad \wedge \quad \re(x) < \omegaeps^2 \quad \wedge \quad \rhoef(x) \geq \omegaeps \quad \forall \bfv \in \mathcal{V}_{\bfe},\; \bff\in \mathcal{F}_{\bfe} \},
\\
    \omegaf^{\omegaeps} &\coloneqq \{x \in \Omega\,:\, \rv(x) \geq \omegaeps \quad \wedge \quad \re(x) \geq \omegaeps^2 \quad \wedge \quad \rf(x) < \omegaeps^3 \quad \forall \bfv \in \mathcal{V}_{\bff}, \; \bfe \in \mathcal{E}_{\bff} \},
    \\
    \Omega_{\mathrm{int}}^{\omegaeps} 
    &\coloneqq \{x \in \Omega\,:\, \rv(x) \geq \omegaeps \quad \wedge \quad \re(x) \geq \omegaeps^2 \quad \wedge \quad 
                                   \rf(x) \geq \omegaeps^3 \quad \forall \bfv, \bfe, \bff \}.
  \end{align*}

\begin{figure}[ht]
\begin{center}
\begin{tikzpicture}[scale=1.2]
  \def\R{3}
  \def\A{90}
\draw ({3/4*\R*cos(\A)-.1},{3/4*\R*sin(\A)}) node[above]{$\mathbf{f}'$};
\draw (0,0) node {\textbullet} node[left] {$\mathbf{e}/\mathbf{v}$}; 
  \draw[-] (0, 0) -- ({\R*cos(\A)}, {\R*sin(\A)});
  \draw[-] (0, 0) -- (\R, 0); 
  \draw[dashed] (0, 0) -- ({2/3*\R*cos(\A/3)},  {2/3*\R*sin(\A/3)}); 
  \draw[dashed] (0, 0) -- ({2/3*\R*cos(2*\A/3)},{2/3*\R*sin(2*\A/3)}); 
  \draw[dashed] ({2/3*\R*cos(\A/3)}, {2/3*\R*sin(\A/3)}) -- (\R,{2/3*\R*sin(\A/3)}); 
  \draw[dashed] %
  ({2/3*\R*cos(2*\A/3)}, {2/3*\R*sin(2*\A/3)}) %
  -- ({2/3*\R*cos(2*\A/3) + \R*(1-2/3*cos(\A/3))*cos(\A)}, {2/3*\R*sin(2*\A/3) + \R*(1-2/3*cos(\A/3))*sin(\A)});
\draw (3/4*\R, 0) node [below]{$\mathbf{f}$} ;
\draw (3/8*\R ,0.05) node [above]{$\omega_{\mathbf{v}\mathbf{e}\mathbf{f}}$} ;
\draw (7/8*\R, 0.15) node [above]{$\omega_{\mathbf{v}\mathbf{f}}$} ;
\draw ({\R*cos(\A/2)}, {\R*sin(\A/2)}) node {$\omega_{\mathbf{v}}$} ;
\draw ({3/8*\R*cos(\A/2)},{3/8*\R*sin(\A/2)-0.12}) node [above]{$\omega_{\mathbf{v}\mathbf{e}}$}; 
\draw ({3/8*\R*cos(3*\A/4)},{3/8*\R*sin(3*\A/4)}) node [above]{$\omega_{\mathbf{v}\mathbf{e}\mathbf{f}'}$} ;
\draw ({7/8*\R*cos(3*\A/4)-0.25},{7/8*\R*sin(3*\A/4)}) node [above]{$\omega_{\mathbf{v}\mathbf{f}'}$} ;
\draw [dashed,domain=0:\A] plot ({2/3*\R*cos(\x)}, {2/3*\R*sin(\x)});
\end{tikzpicture}%
\hspace{8mm}
\begin{tikzpicture}[scale=0.7]
  \def\R{3}
  \def\A{60}
\draw[fill=magenta!10,domain=\A/2:\A] plot ({2/3*\R*cos(\x)}, {2/3*\R*sin(\x)}) -- ({\R*cos(\A)}, {\R*sin(\A)}) -- ({2/3*\R*cos(\A/2) + \R*(1-2/3*cos(\A/2))*cos(\A)}, {2/3*\R*sin(\A/2) + \R*(1-2/3*cos(\A/2))*sin(\A)}) -- cycle;
\draw[fill=magenta!10]  ({2/3*\R*cos(\A)}, {2/3*\R*sin(\A)}) -- ({\R*cos(\A)}, {\R*sin(\A)}) -- ({-\R}, {0}) -- cycle;

\draw[fill=magenta!10,domain=\A/2:\A] plot ({2/3*\R*cos(\x)}, {-2/3*\R*sin(\x)}) -- ({\R*cos(\A)}, {-\R*sin(\A)}) -- ({2/3*\R*cos(\A/2) + \R*(1-2/3*cos(\A/2))*cos(\A)}, {-2/3*\R*sin(\A/2) - \R*(1-2/3*cos(\A/2))*sin(\A)}) -- cycle;
\draw[fill=magenta!10]  ({2/3*\R*cos(\A)}, {-2/3*\R*sin(\A)}) -- ({\R*cos(\A)}, {-\R*sin(\A)}) -- ({-\R}, {0}) -- cycle;

\draw[fill=cyan!10,domain=-\A/2:\A/2] plot ({2/3*\R*cos(\x)}, {2/3*\R*sin(\x)}) -- (0,0) -- ({2/3*\R*cos(-\A/2)}, {2/3*\R*sin(-\A/2)});

\draw[fill=blue!10,domain=\A/2:\A] plot ({2/3*\R*cos(\x)}, {2/3*\R*sin(\x)}) -- (0,0) -- ({2/3*\R*cos(\A/2)}, {2/3*\R*sin(\A/2)});
\draw[fill=blue!10]  ({2/3*\R*cos(\A)}, {2/3*\R*sin(\A)}) -- (0,0) -- (-\R,0) -- ({-\R}, {0}) -- cycle;

\draw[fill=blue!10,domain=\A/2:\A] plot ({2/3*\R*cos(\x)}, {-2/3*\R*sin(\x)}) -- (0,0) -- ({2/3*\R*cos(\A/2)}, {-2/3*\R*sin(\A/2)});
\draw[fill=blue!10]  ({2/3*\R*cos(\A)}, {-2/3*\R*sin(\A)}) -- (0,0) -- (-\R,0) -- cycle;

\draw[dashed] ({-\R}, {0}) -- ({2/3*\R*cos(\A/2)}, {2/3*\R*sin(\A/2)});
\draw[dashed] ({-\R}, {0})  -- ({2/3*\R*cos(-\A/2)}, {2/3*\R*sin(-\A/2)});

\draw[white,thick] ({2/3*\R*cos(\A/2) + \R*(1-2/3*cos(\A/2))*cos(\A)}, {2/3*\R*sin(\A/2) + \R*(1-2/3*cos(\A/2))*sin(\A)}) -- ({\R*cos(\A)}, {\R*sin(\A)}) -- ({\R*cos(\A)-\R}, {\R*sin(\A)})--({-\R}, {0}) -- ({\R*cos(\A)-\R}, {-\R*sin(\A)}) -- ({\R*cos(\A)}, {-\R*sin(\A)})--({2/3*\R*cos(\A/2) + \R*(1-2/3*cos(\A/2))*cos(\A)}, {-2/3*\R*sin(\A/2) - \R*(1-2/3*cos(\A/2))*sin(\A)});

\draw[] (-\R/2+0.2,0) node[below] {$\mathbf{e}$};
\draw[] (-\R/4,-2*\R/3+0.2) node[below] {$\mathbf{f}$};
\draw[] (-\R/4,2*\R/3) node[below] {$\mathbf{f}'$};
\draw[line width=1.5] (-\R,0) -- (0,0);
\draw[line width=1.5] (-\R,0) -- ({\R*cos(\A)-\R}, {\R*sin(\A)});
\draw[line width=1.5] (-\R,0) -- ({\R*cos(\A)-\R}, {-\R*sin(\A)});

\draw[cyan] (\R/3+0.2,0) node[] {$\omega_{\mathbf{ve}}$};
\draw[blue] (\R/4-0.1,\R/4+0.3) node[right] {$\omega_{\mathbf{vef'}}$};
\draw[blue] (\R/4-0.1,-\R/4-0.3) node[right] {$\omega_{\mathbf{vef}}$};

\draw[magenta] (\R/2-0.2,\R/2+0.2) node[right] {$\omega_{\mathbf{vf'}}$};
\draw[magenta] (\R/2-0.2,-\R/2-0.2) node[right] {$\omega_{\mathbf{vf}}$};
\draw (\R-0.5,0)  node {$\omega_{\mathbf{v}}$} ;
\end{tikzpicture}
\caption{\label{fig:vertex-notation} 
Notation near a vertex $\mathbf{v}$. Left: top view of the vertex cone. Right: side view of the vertex cone.}
\end{center}
\end{figure}
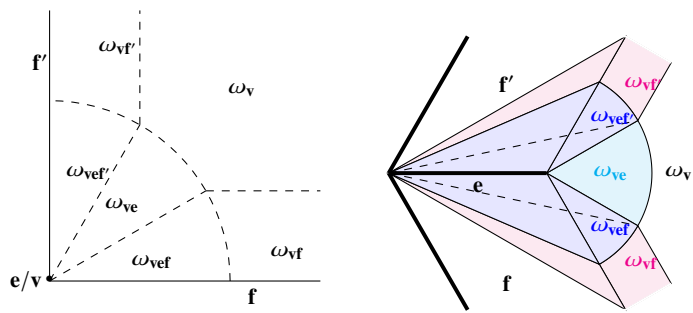

\begin{figure}
\begin{center}
\begin{tikzpicture}[scale=1.2]
  \def\R{3}
  \def\A{90}
\draw ({3/4*\R*cos(\A)-.1},{3/4*\R*sin(\A)}) node[above]{$\mathbf{f}'$};
\draw (0,0) node {\textbullet} node[left] {$\mathbf{e}$}; 
  \draw[-] (0, 0) -- ({\R*cos(\A)}, {\R*sin(\A)});
  \draw[-] (0, 0) -- (\R, 0); 
  \draw[dashed] (0, 0) -- ({2/3*\R*cos(\A/3)},  {2/3*\R*sin(\A/3)}); 
  \draw[dashed] (0, 0) -- ({2/3*\R*cos(2*\A/3)},{2/3*\R*sin(2*\A/3)}); 
  \draw[dashed] ({2/3*\R*cos(\A/3)}, {2/3*\R*sin(\A/3)}) -- (\R,{2/3*\R*sin(\A/3)}); 
  \draw[dashed] %
  ({2/3*\R*cos(2*\A/3)}, {2/3*\R*sin(2*\A/3)}) %
  -- ({2/3*\R*cos(2*\A/3) + \R*(1-2/3*cos(\A/3))*cos(\A)}, {2/3*\R*sin(2*\A/3) + \R*(1-2/3*cos(\A/3))*sin(\A)});
\draw (3/4*\R, 0) node [below]{$\mathbf{f}$} ;
\draw (3/8*\R ,0.05) node [above]{$\omega_{\mathbf{e}\mathbf{f}}$} ;
\draw (7/8*\R, 0.15) node [above]{$\omega_{\mathbf{f}}$} ;
\draw ({\R*cos(\A/2)}, {\R*sin(\A/2)}) node {$\Omega_{\mathrm{int}}$} ;
\draw ({3/8*\R*cos(\A/2)},{3/8*\R*sin(\A/2)-0.12}) node [above]{$\omega_{\mathbf{e}}$}; 
\draw ({3/8*\R*cos(3*\A/4)},{3/8*\R*sin(3*\A/4)}) node [above]{$\omega_{\mathbf{e}\mathbf{f}'}$} ;
\draw ({7/8*\R*cos(3*\A/4)-0.25},{7/8*\R*sin(3*\A/4)}) node [above]{$\omega_{\mathbf{f}'}$} ;
\draw [dashed,domain=0:\A] plot ({2/3*\R*cos(\x)}, {2/3*\R*sin(\x)});
\end{tikzpicture}%
 \hspace{8mm}
\begin{tikzpicture}[scale=0.7]
  \def\R{3}
  \def\A{60}
\draw[fill=magenta!10,domain=\A/2:\A] plot ({2/3*\R*cos(\x)}, {2/3*\R*sin(\x)}) -- ({\R*cos(\A)}, {\R*sin(\A)}) -- ({2/3*\R*cos(\A/2) + \R*(1-2/3*cos(\A/2))*cos(\A)}, {2/3*\R*sin(\A/2) + \R*(1-2/3*cos(\A/2))*sin(\A)}) -- cycle;
\draw[fill=magenta!10,domain=\A/2:\A] plot ({2/3*\R*cos(\x)}, {2/3*\R*sin(\x)}) -- (0,0) -- ({2/3*\R*cos(\A/2)}, {2/3*\R*sin(\A/2)});
\draw[fill=magenta!10]  ({2/3*\R*cos(\A)}, {2/3*\R*sin(\A)}) -- ({\R*cos(\A)}, {\R*sin(\A)}) -- ({\R*cos(\A)-\R}, {\R*sin(\A)}) -- ({2/3*\R*cos(\A)-\R}, {2/3*\R*sin(\A)}) -- cycle;

\draw[fill=magenta!10,domain=\A/2:\A] plot ({2/3*\R*cos(\x)}, {-2/3*\R*sin(\x)}) -- ({\R*cos(\A)}, {-\R*sin(\A)}) -- ({2/3*\R*cos(\A/2) + \R*(1-2/3*cos(\A/2))*cos(\A)}, {-2/3*\R*sin(\A/2) - \R*(1-2/3*cos(\A/2))*sin(\A)}) -- cycle;
\draw[fill=magenta!10,domain=\A/2:\A] plot ({2/3*\R*cos(\x)}, {-2/3*\R*sin(\x)}) -- (0,0) -- ({2/3*\R*cos(\A/2)}, {-2/3*\R*sin(\A/2)});
\draw[fill=magenta!10]  ({2/3*\R*cos(\A)}, {-2/3*\R*sin(\A)}) -- ({\R*cos(\A)}, {-\R*sin(\A)}) -- ({\R*cos(\A)-\R}, {-\R*sin(\A)}) -- ({2/3*\R*cos(\A)-\R}, {-2/3*\R*sin(\A)}) -- cycle;

\draw[fill=cyan!10,domain=-\A/2:\A/2] plot ({2/3*\R*cos(\x)}, {2/3*\R*sin(\x)}) -- (0,0) -- ({2/3*\R*cos(-\A/2)}, {2/3*\R*sin(-\A/2)});

\draw[fill=blue!10,domain=\A/2:\A] plot ({2/3*\R*cos(\x)}, {2/3*\R*sin(\x)}) -- ({2/3*\R*cos(\A)-\R}, {2/3*\R*sin(\A)}) -- ({2/3*\R*cos(\A/2)-\R}, {2/3*\R*sin(\A/2)}) -- ({2/3*\R*cos(\A/2)}, {2/3*\R*sin(\A/2)});
\draw[fill=blue!10,domain=\A/2:\A] plot ({2/3*\R*cos(\x)}, {2/3*\R*sin(\x)}) -- (0,0) -- ({2/3*\R*cos(\A/2)}, {2/3*\R*sin(\A/2)});
\draw[fill=blue!10]  ({2/3*\R*cos(\A)}, {2/3*\R*sin(\A)}) -- (0,0) -- (-\R,0) -- ({2/3*\R*cos(\A)-\R}, {2/3*\R*sin(\A)}) -- cycle;

\draw[fill=blue!10,domain=\A/2:\A] plot ({2/3*\R*cos(\x)}, {-2/3*\R*sin(\x)}) -- ({2/3*\R*cos(\A)-\R}, {-2/3*\R*sin(\A)}) -- ({2/3*\R*cos(\A/2)-\R}, {-2/3*\R*sin(\A/2)}) -- ({2/3*\R*cos(\A/2)}, {-2/3*\R*sin(\A/2)});
\draw[fill=blue!10,domain=\A/2:\A] plot ({2/3*\R*cos(\x)}, {-2/3*\R*sin(\x)}) -- (0,0) -- ({2/3*\R*cos(\A/2)}, {-2/3*\R*sin(\A/2)});
\draw[fill=blue!10]  ({2/3*\R*cos(\A)}, {-2/3*\R*sin(\A)}) -- (0,0) -- (-\R,0) -- ({2/3*\R*cos(\A)-\R}, {-2/3*\R*sin(\A)}) -- cycle;

\draw[dashed] ({2/3*\R*cos(\A/2)-\R}, {2/3*\R*sin(\A/2)}) -- ({2/3*\R*cos(\A/2)}, {2/3*\R*sin(\A/2)});
\draw[dashed] ({2/3*\R*cos(-\A/2)-\R}, {2/3*\R*sin(-\A/2)}) -- ({2/3*\R*cos(-\A/2)}, {2/3*\R*sin(-\A/2)});

\draw[white,thick] ({2/3*\R*cos(\A/2) + \R*(1-2/3*cos(\A/2))*cos(\A)}, {2/3*\R*sin(\A/2) + \R*(1-2/3*cos(\A/2))*sin(\A)}) -- ({\R*cos(\A)}, {\R*sin(\A)}) -- ({\R*cos(\A)-\R}, {\R*sin(\A)})--({-\R}, {0}) -- ({\R*cos(\A)-\R}, {-\R*sin(\A)}) -- ({\R*cos(\A)}, {-\R*sin(\A)})--({2/3*\R*cos(\A/2) + \R*(1-2/3*cos(\A/2))*cos(\A)}, {-2/3*\R*sin(\A/2) - \R*(1-2/3*cos(\A/2))*sin(\A)});

\draw[] (-\R/2,0) node[below] {$\mathbf{e}$};
\draw[] (-\R/4,-\R/3) node[below] {$\mathbf{f}$};
\draw[] (-\R/4,\R/3) node[below] {$\mathbf{f}'$};
\draw[line width=2] (-\R,0) -- (0,0);

\draw[cyan] (\R/3,0) node[] {$\omega_{\mathbf{e}}$};
\draw[blue] (\R/4-0.1,\R/4+0.4) node[right] {$\omega_{\mathbf{ef'}}$};
\draw[blue] (\R/4-0.1,-\R/4-0.4) node[right] {$\omega_{\mathbf{ef}}$};

\draw[magenta] (\R/2,\R/2+0.4) node[right] {$\omega_{\mathbf{f'}}$};
\draw[magenta] (\R/2,-\R/2-0.4) node[right] {$\omega_{\mathbf{f}}$};
\draw (\R-0.5,0)  node {$\Omega_{\mathrm{int}}$} ;
\end{tikzpicture}
\end{center}
\caption{\label{fig:edge-notation}
Notation near an edge $\mathbf{e}$ with two faces $\mathbf{f},\mathbf{f}'$ 
meeting at the edge and no vertex close by. Left: front view (edge collapses to point). Right: side view.}
\end{figure}
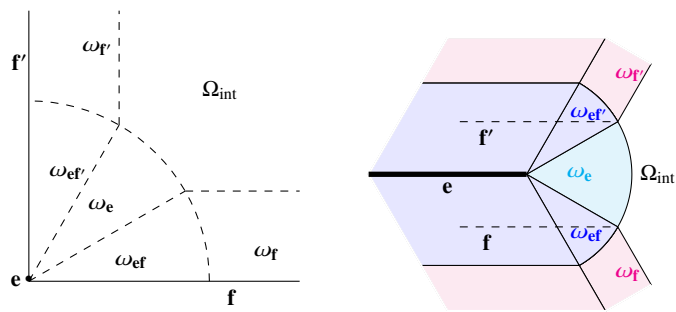

Figure~\ref{fig:vertex-notation} 
illustrates the neighborhoods near a vertex and Figure~\ref{fig:edge-notation} 
shows the neighborhoods close to an edge but away from a vertex.
We drop the superscript $\omegaeps$ unless strictly necessary.

Each neighborhood may have a different value $\omegaeps$, 
but we assume that each $\omega_{\bullet}$ abuts at 
most one vertex, one edge, or one face of $\partial\Omega$. 
Since only finitely many distinct types of neighborhoods are needed to decompose the polyhedron, 
the interior $\Omega_{\rm int}\subset\Omega $ has a positive distance from the boundary.  
\bigskip

In each of the neighborhoods, as done in \cite{FMMS25-regularity3d}, we may choose local coordinates that introduce orthogonal coordinates w.r.t. the closest face/edge. 

  Let $\omega_{\bullet}^\omegaeps \subset \Omega$ with $\bullet \in \{\bfv,\bfe,\bff,\bfv\bfe,\bfv\bff,\bfe\bff,\bfv\bfe\bff\}$ be an arbitrary vertex/edge/face neighborhood. We take orthogonal unit vectors $\{\gperp, \gparperp, \gpar\}$  such that :
\begin{itemize}
    \item  For $\bullet \in \{\bff,\bfv\bff\}$, let $\gpar,\gparperp$ span the plane tangential to $\bff$ and let $\gperp$ be perpendicular to $\bff$.
    \item  For $\bullet \in \{\bfe,\bfv\bfe\}$, let $\gpar$ be tangential to $\bfe$ and $\gparperp, \gperp$ be mutually orthogonal and span the plane transversal to $\bfe$.
    \item  For $\bullet \in \{\bfe\bff,\bfv\bfe\bff\}$, let $\gpar$ be parallel to $\bfe$ and $\bff$, $\gparperp$ be perpendicular to $\bfe$ and parallel to $\bff$ and $\gperp$ be perpendicular to $\bfe$ and $\bff$.
    \item  For $\bullet \in \{\bfv\}$ or the set $\Omega_{\rm int} \subset \Omega$, we take the standard coordinates in $\R^3$.
 \end{itemize}

  With these vectors and for  $\beta = (\betaperp, \betaparperp, \betapar)\in\N^3_0$, we introduce the derivative 
 \begin{align*}
 D_{(\gperp, \gparperp, \gpar)}^{\beta} =  D_{\gperp}^{\betaperp}D_{\gparperp}^{\betaparperp}D_{\gpar}^{\betapar}.
 \end{align*}

For our purpose of studying the regularity on the unit cube, these coordinates can easily be chosen as rotations of the canonical coordinates with the important property that the coordinate $\gperp$ is always orthogonal to the considered (closest) face of the unit cube.

\subsection{Weighted analytic regularity}
\label{sec:StatReg}
We recall the analytic regularity result of \cite{FMMS25-regularity3d} 
for \eqref{eq:FracLap}.
%
\begin{theorem}\label{thm:regularity}
Let $\Omega := (0,1)^3$ and $f \in C^{\infty}(\overline{\Omega})$ satisfy
with a constant $\gamma_f>0$
  \begin{equation}
    \label{eq:analyticdata}
\forall j \in \N_0\colon \quad 
    \sum_{\betam = j} \|\partial^\beta f\|_{L^2(\Omega)} \leq \gamma_f^{j+1}j^j. 
  \end{equation}
Let $u$ be the weak solution of \eqref{eq:weakform}.

Then, 
there exists $\gamma>0$ depending only on $\gamma_f$, $s$, 
and $\Omega$ such that for all $t<1/2$, 
    there exists $C_t>0$ such that for all $\beta = (\betaperp, \betaparperp, \betapar)\in\N^3_0$:
\begin{enumerate}
\item For all $\bfv \in \mathcal{V}$ and neighborhoods $\omega_{\bfv}$, there holds
  \begin{align}\label{eq:weighted_regularity_vertex}
    \| r_{\partial \Omega}^{-t-s} r_{\bfv}^{\betapar}
             D_{(\gperp, \gparperp, \gpar)}^{\beta} u \|_{L^2(\omega_{\bfv})} \leq C_t \gamma^{\betam}\betam^{\betam}.
  \end{align}
\item For all $\bfe \in \mathcal{E}$ and neighborhoods $\omega_{\bfe}$, there holds
  \begin{align}\label{eq:weighted_regularity_edge}
    \| r_{\partial \Omega}^{-t-s} r_{\bfe}^{\betaparperp}
             D_{(\gperp, \gparperp, \gpar)}^{\beta} u \|_{L^2(\omega_{\bfe})} \leq C_t \gamma^{\betam}\betam^{\betam}.
  \end{align}
\item For all $\bff \in \mathcal{F}$ and neighborhoods $\omega_{\bff}$, there holds
  \begin{align}\label{eq:weighted_regularity_face}
    \| r_{\partial \Omega}^{-t-s} r_{\bff}^{\betaperp}
             D_{(\gperp, \gparperp, \gpar)}^{\beta} u \|_{L^2(\omega_{\bff})} \leq C_t \gamma^{\betam}\betam^{\betam}.
  \end{align}
\item For all $\bfv \in \mathcal{V}, \bfe \in \mathcal{E}$ and neighborhoods $\omega_{\bfv\bfe}$, there holds
  \begin{align}\label{eq:weighted_regularity_vertexedge}
    \| r_{\partial \Omega}^{-t-s} r_{\bfv}^{\betapar}r_{\bfe}^{\betaparperp}
             D_{(\gperp, \gparperp, \gpar)}^{\beta} u \|_{L^2(\omega_{\bfv\bfe})} \leq C_t \gamma^{\betam}\betam^{\betam}.
  \end{align}
\item For all $\bfv \in \mathcal{V}, \bff \in \mathcal{F}$ and neighborhoods $\omega_{\bfv\bff}$, there holds
  \begin{align}\label{eq:weighted_regularity_vertexface}
    \| r_{\partial \Omega}^{-t-s} r_{\bfv}^{\betapar}r_{\bff}^{\betaperp}
             D_{(\gperp, \gparperp, \gpar)}^{\beta} u \|_{L^2(\omega_{\bfv\bff})} \leq C_t \gamma^{\betam}\betam^{\betam}.
  \end{align}
\item For all $\bfe \in \mathcal{E}, \bff \in \mathcal{F}$ and neighborhoods $\omega_{\bfe\bff}$, there holds
  \begin{align}\label{eq:weighted_regularity_edgeface}
    \| r_{\partial \Omega}^{-t-s} r_{\bfe}^{\betaparperp}r_{\bff}^{\betaperp}
             D_{(\gperp, \gparperp, \gpar)}^{\beta} u \|_{L^2(\omega_{\bfe\bff})} \leq C_t \gamma^{\betam}\betam^{\betam}.
  \end{align}
\item For all $\bfv \in \mathcal{V}, \bfe \in \mathcal{E}, \bff \in \mathcal{F}$ and neighborhoods $\omega_{\bfv\bfe\bff}$, there holds
  \begin{align}\label{eq:weighted_regularity_vertexedgeface}
    \| r_{\partial \Omega}^{-t-s} r_{\bfv}^{\betapar} r_{\bfe}^{\betaparperp} r_{\bff}^{\betaperp} 
             D_{(\gperp, \gparperp, \gpar)}^{\beta} u \|_{L^2(\omega_{\bfv\bfe\bff})} \leq C_t \gamma^{\betam}\betam^{\betam}.
  \end{align}
\item For $\Omega_{\rm int}$ there holds 
  \begin{align}\label{eq:weighted_regularity_interior}
    \| r_{\partial \Omega}^{-t-s} D_{(\gperp, \gparperp, \gpar)}^{\beta} u \|_{L^2(\Omega_{\rm int})} \leq C_t \gamma^{\betam}\betam^{\betam}.
  \end{align}
\end{enumerate}
\end{theorem}

\begin{remark}
As we are considering the unit cube, the use of axis parallel, tensor product elements and local coordinates that are simple rotations of the canonical coordinates significantly simplifies the proof of the main result. In the case of general polyhedra, more complicated, and in general non-affine element transformations and changes of variables have to be taken into account, which induces non-trivial change of variable transformations in the regularity theory. Incorporating this will be the topic of the follow-up work \cite{FMMSfuture}.
\end{remark}
\section{Exponential convergence of $hp$-FEM}
\label{sec:proof}

In this section, we provide the proof of our main result, Theorem~\ref{thm:hpExpConv}. Starting from the C\'ea-Lemma, we have to construct 
a function $v_N \in W^L_q = \mathcal{S}^q_0(\Omega,\Tg)$ that induces the claimed root exponential bound in the $\widetilde H^s(\Omega)$-norm. In particular, 
we will exploit that elements in $\Tg$ close to the boundary are exponentially small and use a $hp$-FEM interpolation operator $\Pi_q^L$ on the remaining elements.

A major difficulty in the analysis of fractional PDEs is the non-local nature of the energy norm, which can be addressed by either taking a localizable upper bound, \cite{Faermann2002}, or by using an embedding into a weighted integer order Sobolev space. Here, in the same way as in our previous works, \cite{BFMMS-hp1d,FMMS23-hp2d}, we choose the latter way of localization. 

\subsection{Embedding into weighted integer order space}
\label{sec:LoclNrm}

Recall $r_{\partial\Omega}(x):=\operatorname{dist}(x,\partial\Omega)$ for $x\in \Omega$.
For $\mu \in [0,1)$ 
and an open set $\omega\subseteq \Omega$,
denote by 
$H^1_\mu(\omega)$ the local Sobolev space 
defined via the weighted norm 
$\| \cdot \|_{H^1_\mu(\omega)}$ given by
\begin{align}\label{eq:H1b}
\| v \|^2_{H^1_\mu(\omega)}
:= 
\| r_{\partial\Omega}^\mu \nabla v \|_{L^2(\omega)}^2
+ 
\| r_{\partial\Omega}^{\mu-1} v    \|_{L^2(\omega)}^2
.
\end{align}
By \cite[Prop.~3.1]{FMMS23-hp2d}, one has the embedding 
$H^1_\mu(\Omega) \subset \widetilde H^s(\Omega)$ for $\mu \in [0,1-s)$, 
i.e., there holds 
\begin{align}\label{eq:LocNrm}
\forall v\in H^1_\mu(\Omega): \quad 
\| v \|_{\tH^{s}(\Omega)}
\leq 
C_{\mu,\sigma} \| v \|_{H^1_\mu(\Omega)} 
\;.
\end{align}
Thus, rather than working with the non-local $\tH^s$-norm, we can construct localized approximations in the local $H^1_\mu$-norm.

\subsection{Definition of the $hp$-interpolation operator  $\Pi_q^L$}
\label{S:DefhpInt}

Due to the localization of the previous subsection, we aim  for an elementwise construction of the 
global $hp$-interpolator $\PiLq: H^1_\mu(\Omega) \to \VLq$, which will be obtained 
by assembling local Gauss-Lobatto-Legendre (GLL) interpolants.

The tensor product GLL interpolants on the reference cube approximate analytic functions 
on the reference cube with exponential accuracy in the employed polynomial degree. 
The following Lemma~\ref{lemma:hat_Pi_1_infty} is a 
straight-forward generalization of the case $d = 2$ from 
\cite[Lem.~{3.2}]{BMS23} to the present case $d = 3$: 

\begin{lemma}[approximation on cuboids]
\label{lemma:hat_Pi_1_infty}
Let $\widehat K = (0,1)^3$ 
be the reference cube. 
For each $q \in {\mathbb N}$,
the tensor-product Gauss-Lobatto interpolation operator 
$\widehat \Pi_q := \hat{\pi}_q^{\otimes 3}: 
  C^0(\overline{\widehat K}) \rightarrow {\bbQ}_q$,
with $\hat{\pi}_q$ being the Gauss-Lobatto interpolant of polynomial degree $q\geq 1$ 
    on $(0,1)$, satisfies (with the projectors understood to act on functions 
    defined on $\hat{e}$, $\hat{f}$, 
    identified with $(0,1)$ and $(0,1)^2$, respectively):
\begin{enumerate}
\item[0]
\label{item:lemma:hat_Pi_1:point}
For each vertex $\hat{v} \subset \partial\widehat K $,
$(\widehat \Pi_q u)(\hat{v}) = u(\hat{v})$.
\item[1]
\label{item:lemma:hat_Pi_1_infty-0a}
For each edge
$\hat{e} \subset \partial\widehat K$, 
$(\widehat \Pi_q u)|_{\hat{e}} = \hat{\pi}_q(u|_{\hat{e}})$.
\item[2]
\label{item:lemma:hat_Pi_1_infty-0b}
For each face 
$\hat{f} \subset \partial\widehat K$, 
$(\widehat \Pi_q u)|_{\hat{f}} = \hat{\pi}_q^{\otimes 2}(u|_{\hat{f}})$.
%
%
\item [$\pi$] (projection property)
\label{item:lemma:hat_Pi_1_infty--1}
$(\widehat \Pi_q)^2 = \widehat \Pi_q$ on ${\bbQ}_q = \bbP_q^{\otimes 3}$. 
\item [$\alpha$] (approximation property)
\label{item:lemma:hat_Pi_1_infty-ii}
Let $u \in C^\infty(\widehat K)$ satisfy 
for some $C_u$, $\gamma > 0$, 
and all
$(n,m,\ell) \in {\mathbb N}_0^3 $
\begin{align}
\label{eq:lemma:hat_Pi_1_infty-reg}
\|\partial_\xh^m \partial_\yh^n \partial_\zh^\ell u\|_{L^\infty(\widehat K)} 
\leq 
C_u \gamma^{n+m+\ell}
(m+1)^m (n+1)^n  (\ell +1)^\ell.
\end{align}
Then, there exist $C, b>0$ such that for all $q\geq 1$
 \begin{align*}
\|u - \widehat \Pi_q u\|_{W^{1,\infty}(\widehat K)}  \leq C C_u e^{-bq}.
 \end{align*}

\end{enumerate}
\end{lemma}

The global $hp$-interpolator $\Pi^L_q$
is assembled from elementwise projectors 
via
\begin{align*}
\forall K \in \Tg: \; 
(\Pi^L_q u)|_K \circ F_K  
:= 
\widehat\Pi_q (\hat{u}_K),
\end{align*}
where $\hat{u}_K$ is as in \eqref{eq:hatu}.
    By items 1. and 2., $\Pi^L_q: C^0(\overline{\Omega})\to S^q_0(\Omega,\Tg)$.
\subsection{Mesh layers and cutoff function}
\label{sec:LayCut}
For $L\in \mathbb{N}$, we subdivide the mesh $\Tg$ into
boundary layer $\lay_0$, 
transition layer $\lay_1$, and 
internal mesh elements $\lay_{\mathrm{int}}$. 
Specifically, we let
\begin{align*}
  \lay_0 &:= \left\{ K\in \Tg: \overline{K}\cap \partial\Omega \neq \emptyset \right\},\\
  \lay_1 &:= \left\{ K\in \Tg\setminus \lay_0: \exists\, J\in \lay_0 \text{ such that }
  \overline{K}\cap \overline{J} \neq \emptyset \right\},\\
  \lay_{\mathrm{int}} &:= \Tg \setminus\left( \lay_0\cup\lay_1 \right).
\end{align*}
We introduce the continuous
cutoff function $\cutoff:\Omega \to [0,1]$ satisfying
\begin{equation}
  \label{eq:cutoff-prop}
\cutoff \in \mathcal{S}^1_0(\Omega, \Tg),
\qquad
\cutoff \equiv 0\text{ on all }K\in \lay_0,
\qquad
\cutoff \equiv 1\text{ on all }K\in \lay_{\mathrm{int}}.
\end{equation}
The subdomain consisting of the union of all mesh elements touching $\partial\Omega$ is
\begin{equation}
  \label{eq:SLzero}
\Omega^L_0 = {\rm interior} \left( \bigcup_{K\in \lay_0} \overline{K} \right)\;.
\end{equation}

Using the cut-off function $\cutoff$, 
we can split the best-approximation error into a contribution away from the boundary 
and a term that is only non-zero on the (small) boundary and transition layers:
\begin{align}
    \label{eq:error-decomp}
      \inf_{v\in \VLq}\| u - v \|_{\tH^s(\Omega)} &\leq \inf_{v\in \VLq}
\| \cutoff u - v \|_{\tH^s(\Omega)}  + \| (1-\cutoff) u \|_{\tH^s(\Omega)} \nonumber
\\
&\leq C_{\mu, s}
\| \cutoff (u - \Pi^L_{q-1}) u\|_{H^1_\mu(\Omega)}  + \| (1-\cutoff) u \|_{\tH^s(\Omega)}.
\end{align}
Here, we used $\cutoff\in \mathcal{S}^1_0(\Omega, \Tg)$ 
so that $\cutoff \Pi^L_{q-1} u \in \VLq$ for $q \geq 2$.

By \cite[Lem.~B.1, Lem.~B.2]{FMMS23-hp2d}, for $w \in H^1_\mu(\Omega)$, we have the bounds
 \begin{equation}\label{eq:est-cutoff}
 \| \cutoff w\|_{H^1_\mu(\Omega)}        \leq C \| w\|_{H^1_\mu (\Omega\setminus \Omega^L_0 )}, \;\;
    \| (1-\cutoff) w\|_{H^1_\mu(\Omega)} \leq C \| w\|_{H^1_\mu(S_{c\sigma^L})},
  \end{equation}
with constants $c,C>0$ independent of $L$.
Here, for $t>0$,  $ S_t = \{x\in \Omega: r(x)< t\} $.
%
\subsection{Estimates for elements at the boundary}
The last term in \eqref{eq:error-decomp} is exponentially small by construction of $\Tg$.
\begin{lemma}
  \label{lemma:u-gu}
  Let $u$ be the solution to \eqref{eq:weakform} for $s\in (0,1)$.
  Let $L\in \mathbb{N}$ and $\cutoff$ be
defined as in \eqref{eq:cutoff-prop}. 
Then, there exist $C, b>0$ independent of $L$ such that 
\begin{equation}
  \label{eq:u-gu}
  \| u - \cutoff u \|_{\tH^s(\Omega)} \leq C\exp(-bL).
\end{equation}
\end{lemma}
\begin{proof}
We fix $\mu \in [0,1)$ additionally satisfying $\mu\in (1/2-s, 1-s)$ 
and estimate the $H^1_\mu(\Omega)$-norm of $u-\cutoff u$, which gives the stated bound due to \eqref{eq:LocNrm}.

We now decompose $S_{c\sigma^L} $ into its components belonging to vertex, edge,
vertex-edge, and internal neighborhoods, i.e., we consider $\Omega^{\rm int}\cap S_{c\sigma^L}$ as well as $\omega_\bullet \cap S_{c\sigma^L}$ with $\bullet \in \{\bfv,\bfe,\bff,\bfv\bfe,\bfv\bff,\bfe\bff,\bfv\bfe\bff\}$. The union of these sets over all vertices, edges and faces as in \eqref{eq:Nghbrhoods} gives  back the neighborhood $S_{c\sigma^L}$.

Vertex, edge and face neighborhoods:
since $\mu > 1/2-s$, we may choose $t \in [0,1/2)$ such that
$\mu+t+s-1 > 0$. Let $\mathbf{v}\in \mathcal{V}$, $\mathbf{e}\in \mathcal{E}$, $\mathbf{f}\in \mathcal{F}$ and $\omega_\bullet$ be a fixed neighborhood for $\bullet \in \{\bfv,\bfe,\bff\}$. On $\omega_\bullet$, we have $r_{\partial\Omega} \sim r_{\bullet}$. Consequently, the weighted regularity estimate from Theorem~\ref{thm:regularity} implies
\begin{align*}
  &\|u \|_{H^1_{\mu}(\omega_\bullet\cap S_{c\sigma ^L})} \simeq
  \|r_{\bullet}^{\mu-1} u \|_{L^2(\omega_\bullet\cap S_{c\sigma ^L})} + \sum_{\beta \in \N_0^3, |\beta| = 1}
  \|r_{\bullet}^\mu D_{(\gperp, \gparperp, \gpar)}^{\beta}  u \|_{L^2(\omega_\bullet\cap S_{c\sigma ^L})} \\
   & \leq
    \|r_{\bullet}^{-t-s+\mu-1+t+s} u \|_{L^2(\omega_\bullet\cap S_{c\sigma ^L})} +
\sum_{\beta \in \N_0^3, |\beta| = 1}
  \|r_{\bullet}^{-t-s+\mu+t+s-1}r_\bullet D_{(\gperp, \gparperp, \gpar)}^{\beta}  u \|_{L^2(\omega_\bullet\cap S_{c\sigma ^L})}
\\  & \leq
  (c\sigma^L)^{\mu+t+s-1} \bigg( \|r_{\partial\Omega}^{-t-s} u \|_{L^2(\omega_\bullet\cap S_{c\sigma ^L})}
    \\ &  \qquad +  \sum_{\beta \in \N_0^3, |\beta| = 1}
  \|r_{\partial\Omega}^{-t-s} r_\bullet D_{(\gperp, \gparperp, \gpar)}^{\beta}  u \|_{L^2(\omega_\bullet\cap S_{c\sigma ^L})} \bigg)\\
  & \stackrel{Thm.~\ref{thm:regularity}}{\leq} C  (c\sigma^L)^{\mu+t+s-1}.
\end{align*}
The same argument can be employed for the interior part $\Omega^{\rm int}$.

Vertex-edge neighborhoods $\omegave$: Let $\mathbf{v}\in \mathcal{V}$, $\mathbf{e}\in \mathcal{E}$ and $\omegave$ be a fixed vertex-edge neighborhood. On $\omegave$, we have $r_{\partial\Omega}  \sim \re \leq \rv$. Using the regularity estimate \eqref{eq:weighted_regularity_vertexedge}
with $\beta = (\betapar, \betaperp,  \betaparperp) = (0,0,0)$ as well as $\beta = (1,0,0)$,  $\beta = (0,1,0)$  and $\beta = (0,0,1)$, we obtain
\begin{align*}
&\|u \|_{H^1_{\mu}(\omegave\cap S_{c\sigma ^L})} \simeq
  \|r_{\bfe}^{\mu-1} u \|_{L^2(\omegave\cap S_{c\sigma ^L})} + \sum_{\beta \in \N_0^3, |\beta| = 1}
  \|r_{\bfe}^\mu D_{(\gperp, \gparperp, \gpar)}^{\beta}  u \|_{L^2(\omegave\cap S_{c\sigma ^L})}  \\
   & \leq
    \|r_{\mathbf{e}}^{-t-s+\mu-1+t+s} u \|_{L^2(\omegave\cap S_{c\sigma ^L})} +    
   \|r_{\mathbf{e}}^{-t-s+\mu+t+s-1}\re  D_{\gperp} u \|_{L^2(\omegave\cap S_{c\sigma ^L})} 
     \\ & \qquad \qquad + \|r_{\mathbf{e}}^{-t-s+\mu+t+s-1}\re  D_{\gparperp} u \|_{L^2(\omegave\cap S_{c\sigma ^L})} 
 +
\|r_{\mathbf{e}}^{-t-s+\mu+t+s}  \re^{-1} \rv D_{\gpar} u \|_{L^2(\omegave\cap S_{c\sigma ^L})}\\
  & \leq
  (c\sigma^L)^{\mu+t+s-1} \bigg( \|r_{\partial\Omega}^{-t-s} u \|_{L^2(\omegave\cap S_{c\sigma ^L})} + \|r_{\partial\Omega}^{-t-s}\re D_{\gperp} u \|_{L^2(\omegave\cap S_{c\sigma ^L})} 
  \\ &  \qquad
       + \|r_{\partial\Omega}^{-t-s} r_{\mathbf{e}} D_{\gparperp} u \|_{L^2(\omegave\cap S_{c\sigma ^L})} +  \|r_{\partial\Omega}^{-t-s} \rv  D_{\gpar} u \|_{L^2(\omegave\cap S_{c\sigma ^L})} \bigg)\\
  & \stackrel{\eqref{eq:weighted_regularity_vertexedge}}{\leq} C  (c\sigma^L)^{\mu+t+s-1}.
\end{align*}
The same argument holds for vertex-face neighborhoods $\omegavf$, replacing $\re$ with $\rf$, and edge-face neighborhoods, replacing $\re$ with $\rf$ and $\rv$ with $\re$.

Vertex-edge-face neighborhoods $\omegavef$:
Let $\mathbf{v}\in \mathcal{V}$, $\mathbf{e}\in \mathcal{E}$, $\mathbf{f}\in \mathcal{F}$ and $\omegavef$ be a fixed vertex-edge-face neighborhood.
On $\omegavef$, we have $r_{\partial\Omega} \sim \rf \leq \re  \leq \rv$.  Similarly to the previous cases, we obtain
\begin{align*}
&\|u \|_{H^1_{\mu}(\omegavef\cap S_{c\sigma ^L})} \simeq
  \|r_{\bff}^{\mu-1} u \|_{L^2(\omegavef\cap S_{c\sigma ^L})} + \sum_{\beta \in \N_0^3, |\beta| = 1}
  \|r_{\bff}^\mu D_{(\gperp, \gparperp, \gpar)}^{\beta}  u \|_{L^2(\omegavef\cap S_{c\sigma ^L})}  \\
   & \leq
    \|r_{\mathbf{f}}^{-t-s+\mu-1+t+s} u \|_{L^2(\omegavef\cap S_{c\sigma ^L})} +    \|r_{\mathbf{f}}^{-t-s+\mu+t+s-1}\rf  D_{\gperp} u \|_{L^2(\omegavef\cap S_{c\sigma ^L})} 
     \\ & \qquad \qquad +
   \|r_{\mathbf{f}}^{-t-s+\mu+t+s}\rf^{-1}\re  D_{\gparperp} u \|_{L^2(\omegavef\cap S_{c\sigma ^L})}  +
     \|r_{\mathbf{f}}^{-t-s+\mu+t+s}  \rf^{-1} \rv D_{\gpar} u \|_{L^2(\omegavef\cap S_{c\sigma ^L})}\\
  & \leq
  (c\sigma^L)^{\mu+t+s-1} \bigg( \|r_{\partial\Omega}^{-t-s} u \|_{L^2(\omegavef\cap S_{c\sigma ^L})} + \|r_{\partial\Omega}^{-t-s}\rf D_{\gperp} u \|_{L^2(\omegavef\cap S_{c\sigma ^L})} 
  \\ &  \qquad
        + \|r_{\partial\Omega}^{-t-s} r_{\mathbf{e}} D_{\gparperp} u \|_{L^2(\omegavef\cap S_{c\sigma ^L})} +  \|r_{\partial\Omega}^{-t-s} \rv  D_{\gpar} u \|_{L^2(\omegavef\cap S_{c\sigma ^L})}  \bigg)\\
  & \stackrel{\eqref{eq:weighted_regularity_vertexedgeface}}{\leq} C  (c\sigma^L)^{\mu+t+s-1}.
\end{align*}
%
We have thus obtained that  
\begin{equation*}
  \| (1-\cutoff) u \|_{H^1_\mu(\Omega)} 
  \leq  C \| u \|_{H^1_\mu(S_{c\sigma ^L})} 
  \leq C' \exp(-bL),
\end{equation*}
which concludes the proof.
\qed
\end{proof}

\subsection{Proof of \eqref{eq:hpExpConv}}
In this section, 
we prove exponential convergence of $hp$-FEM for \eqref{eq:FracLap}, \eqref{eq:analyticdata}.
\begin{proof}[of Theorem~\ref{thm:hpExpConv}]:
Let $K$  be an axis parallel rectangle and denote by $h_{\parallel, K}$, $h_{\parperp, K}$ and  $h_{\perp, K}$ the sizes of the rectangle $K$ in the coordinate directions $\{\gpar, \gparperp, \gperp\}$. 
To ease notation,
we oftentimes drop the subscript $|_{\cdot, K}$ in the following.

Combining \eqref{eq:error-decomp}, \eqref{eq:est-cutoff} and \eqref{eq:u-gu}, 
we obtain
\begin{align*}
      \inf_{v\in \VLq}\| u - v \|_{\tH^s(\Omega)}^2 
&\leq C \left(
\| \cutoff (u - \Pi^L_{q-1}) u\|_{H^1_\mu(\Omega)}^2  + \| (1-\cutoff) u \|_{\tH^s(\Omega)}^2 \right) \\
&\leq C \left( \| (u - \Pi^L_{q-1}) u\|_{H^1_\mu(\Omega \backslash \Omega^L_0)}^2 + \exp(-2bL) \right) \\
&\leq C \left(\sum_{K \in \Tg \backslash \lay_0} \| (u - \Pi^L_{q-1}) u\|_{H^1_\mu(K)}^2+ \exp(-2bL) \right) .
\end{align*}
We now estimate the approximation error elementwise using 
Theorem~\ref{thm:regularity} and Lemma~\ref{lemma:hat_Pi_1_infty}.

\underline{Step 1:}  [Anisotropic scaling argument in the weighted $H^1$-norm]
The key observation here is that, by construction, 
the geometric mesh has the property that for all $K \in \Tg \backslash \lay_0$ there holds
\begin{align*}
\rv \simeq h_\parallel, \quad \re \simeq h_\parperp, \quad \rf \simeq h_\perp, \qquad r_{\partial\Omega} \simeq h_\perp.
\end{align*}
Using the element transformation $F_K$ such that the coordinates $\{\gpar, \gparperp, \gperp\}$ 
are rotated to the canonical coordinates and recalling $\widehat u = u \circ F_K$, 
we obtain by anisotropic scaling
\begin{align}\label{eq:aniso_scaling}
&\|  r_{\partial\Omega}^{\mu-1}(u - \Pi^L_q u) \|^2_{L^2(K)}
+
\| r_{\partial\Omega}^\mu\nabla(u - \Pi^L_qu) \|^2_{L^2(K)}
\nonumber \\ &\qquad  \leq C
    h_\perp^{2\mu}
      \bigg(   \frac{h_\parallel  h_\parperp}{h_\perp}  \left( \| \hu -
        \hPi_q \hu \|^2_{L^2(\Kh)} + \| \hpartial_z(\hu-\hPi_q\hu)\|^2_{L^2(\Kh)}
      \right) 
  \nonumber    \\ 
&\qquad\quad +  \frac{h_\parallel h_\perp }{h_\parperp} \| \hpartial_y (\hu-\hPi_q\hu)\|^2_{L^2(\Kh)} +  \frac{h_\parperp h_\perp }{h_\parallel } \| \hpartial_x (\hu-\hPi_q\hu)\|^2_{L^2(\Kh)} \bigg).
    \end{align}

\underline{Step 2:} [Approximation on elements inside subregions] 
Let
$K \in \Tg \backslash \lay_0$ and assume $K \subset\omega_{\bullet}^\xi$ with $\bullet \in \{\bfv,\bfe,\bff,\bfv\bfe,\bfv\bff,\bfe\bff,\bfv\bfe\bff\} \cup\{\Omega_{\rm int}\}$. Then, as in \eqref{eq:aniso_scaling}, there holds the scaling estimate for the weighted regularity in vertex-edge-face neighborhoods (the other neighborhoods produce essentially the same estimates)
\begin{align*}
 h_\perp^{1/2-t-s}h_\parperp^{1/2}h_\parallel^{1/2}  \|\hpartial^\beta \hu\|_{L^2(\widehat K)} \simeq    \| r_{\partial \Omega}^{-t-s} r_{\bfv}^{\betapar} r_{\bfe}^{\betaparperp} r_{\bff}^{\betaperp} 
             D_{(\gperp, \gparperp, \gpar)}^{\beta}u \|_{L^2(K)} \leq C \gamma^{\betam}\betam^{\betam}. 
\end{align*}
Together with the Sobolev embedding of $H^2(\Omega)$ into $L^\infty(\Omega)$, this gives the existence of $\widetilde C,\widetilde \gamma >0$, such that
\begin{align*}
\|\hpartial^\beta \hu\|_{L^\infty(\widehat K)} \leq \widetilde C  h_\perp^{-1/2+t+s}h_\parperp^{-1/2}h_\parallel^{-1/2}   \widetilde{\gamma}^{\betam+1}\betam^{\betam}. 
\end{align*}
Now, using Lemma~\ref{lemma:hat_Pi_1_infty} in \eqref{eq:aniso_scaling}, we obtain
\begin{align}\label{eq:aniso_approx}
&\|  r_{\partial\Omega}^{\mu-1}(u - \Pi^L_q u) \|^2_{L^2(K)}
+
\| r_{\partial\Omega}^\mu\nabla(u - \Pi^L_qu) \|^2_{L^2(K)}
\nonumber \\ &\qquad\qquad  \leq C
    h_\perp^{2\mu+2t+2s} \left( h_\parallel^{-2} + h_\parperp^{-2} + h_\perp^{-2}  \right) \exp(-b|\beta|).
\end{align}

\underline{Step 3:} [Treatment of elements intersecting multiple neighborhoods] 
Let
$K \in \Tg \backslash \lay_0$ with $K  \not\subset \omega_{\bullet}^\xi$ for $\bullet \in \{\bfv,\bfe,\bff,\bfv\bfe,\bfv\bff,\bfe\bff,\bfv\bfe\bff\} \cup\{\Omega_{\rm int}\}$.  
Then, by adjusting $\zeta$ (dependent on $\sigma$, but independent on $L$), we may find a neighborhood $\omega_{\bullet}^{\widetilde \xi}$ such that $K \subset \omega_{\bullet}^{\widetilde \xi}$. Consequently, we can employ Theorem~\ref{thm:regularity} for $\omega_{\bullet}^{\widetilde \xi}$  and again obtain estimate \eqref{eq:aniso_approx}. 

\underline{Step 4:} [Global estimates by summation of scaled local bounds] 
By construction of the geometric mesh, 
we always have $h_\perp \leq h_\parperp \leq h_\parallel$. 
Thus, we can choose $\mu <1-s$ and $t < 1/2$ such that $\mu + t >1-s$, which implies
\begin{align*}
   h_\perp^{2\mu+2t+2s} \left( h_\parallel^{-2} + h_\parperp^{-2} + h_\perp^{-2}  \right) \leq h_\perp^{\delta}
\end{align*} 
for some $\delta >0$. 

By symmetry of the unit cube and the constructed mesh, it suffices to do the summation process on the subcube $[0,\frac{1}{2}]^3$. This subcube can be further decomposed into 6 neighborhoods of vertex-edge-face type separating the boundary faces $\{x=0\}, \{y=0\}, \{z=0\}$ (see Figure~\ref{fig:vertex-notation}, where $\xi$ is suitable chosen such that no other neighborhoods are present). The element sizes in perpendicular direction of elements intersecting one of those neighborhoods $\omega_{\mathbf{vef}}$ can be bounded by
\begin{align*}
\sum_{T \in \Tg \backslash \lay_0,T\cap\omega_{\mathbf{vef}} \neq \emptyset}h_{\perp,K}^{\delta} &\leq C\sum_{i=1}^L \sum_{j=i}^L\sum_{k=j}^L \sigma^{\delta k} = C \sigma^{\delta}\frac{(1-\sigma^{\delta L})^3}{(1-\sigma^\delta)^3} 
\\ 
&\lesssim \frac{1}{(1-\sigma^\delta)^3}  =: C_{\sigma,\delta}.
\end{align*}
Thus, we obtain
\begin{align*}
\sum_{K \in \Tg \backslash \lay_0} \| (u - \Pi^L_{q-1}) u\|_{H^1_\mu(K)}^2 
\lesssim \sum_{K \in \Tg \backslash \lay_0}h_{\perp,K}^\delta \exp(-b|\beta|) 
\leq \widetilde C_\sigma \exp(-b|\beta|).
\end{align*}
\end{proof}

\section{Numerical example}
\label{sec:numerics}
The following numerical example 
indicates sharpness of \eqref{eq:hpExpConv} and 
gives quantitative bounds on the dependence of the constant $b$ in \eqref{eq:hpExpConv} on $\sigma$.

We choose $f=1$, $s = 0.25$ 
and employ geometric meshes with different grading parameters $\sigma$ and 
$L$ layers of refinement as well as $hp$-FEM spaces with piecewise polynomials of degree $q=L$. 
We estimate the energy norm error $\sqrt{a(u - u_N,u - u_N)} = \sqrt{a(u,u) - a(u_N,u_N)}$ 
with an extrapolation estimate of $a(u,u)$.

In Figure~\ref{fig:cube-sigma-convergence}, we observe, in accordance with our main result, exponential convergence with respect to the number of layers $L$. 

We note that the stiffness matrix is computed using high precision quadrature formulas based on Duffy transformations to take care of singularities, which is rather expensive in three dimensions. More remarks on implementation and a precise description and analysis of the quadrature rules will be given in a forthcoming paper, \cite{BFM26}.

\begin{figure}
\begin{center}
\begin{tikzpicture}
  \begin{axis}[
    width=0.8\textwidth,
    height=5.5cm,
    ymode=log,
    grid=major,
    grid style={dashed, gray!30},
    axis background/.style={fill=white},
    xlabel={number of layers $L$},
    ylabel={energy norm error},
    xtick={1,2,3},
    legend cell align={left},
    legend style={at={(0.02,0.02)},anchor=south west,font=\footnotesize},
  ]

    \addplot [blue, thick, mark=*]
      table [x=L, y=sigma075, col sep=comma]
      {data_sigma_convergence.txt};
    \addlegendentry{$\sigma = 0.75$}

    \addplot [red, thick, mark=square*]
      table [x=L, y=sigma05, col sep=comma]
      {data_sigma_convergence.txt};
    \addlegendentry{$\sigma = 0.5$}

    \addplot [green!60!black, thick, mark=triangle*]
      table [x=L, y=sigma025, col sep=comma]
      {data_sigma_convergence.txt};
    \addlegendentry{$\sigma = 0.25$}

    \addplot [orange, thick, mark=diamond*]
      table [x=L, y=sigma0172, col sep=comma]
      {data_sigma_convergence.txt};
    \addlegendentry{$\sigma = 0.172$}
    
    \addplot [black, thick, dashed, domain=1:3, samples=100]
      {2.8*pow(0.172, x/2)} node [below,xshift=-0.0cm,yshift=0.05cm] {$0.172^{L/2}$};

    \addplot [black, thick, dashed, domain=1:3, samples=100]
     {2.8*pow(0.5, x/2)} node [below,xshift=-0.0cm,yshift=0.05cm] {$0.5^{L/2}$};



  \end{axis}
\end{tikzpicture}
\caption{
Exponential convergence of $hp$-FEM for \eqref{eq:FracLap} with $f\equiv 1$,
$L=1,2,3$ and $\sigma\in \{(\sqrt{2}-1)^2 = 0.172...,0.25,0.5,0.75\}$.}  
\label{fig:cube-sigma-convergence}
\end{center}
\end{figure}
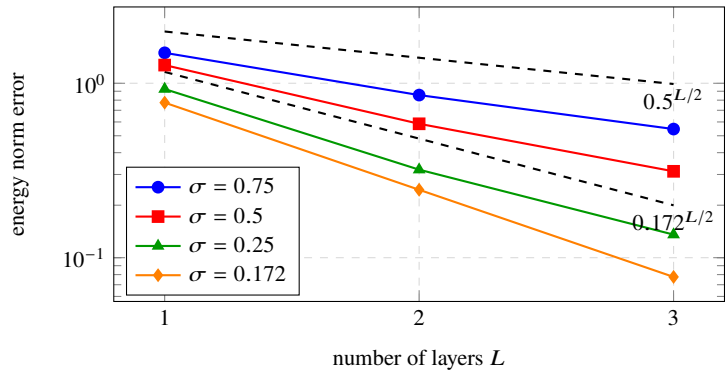
\section{Conclusion}
\label{sec:Concl}
For the homogeneous Dirichlet problem of the 
integral fractional Laplacian in a cube $(0,1)^3$ with analytic in $[0,1]^3$ forcing,
we developed a proof of (root-)exponential convergent $hp$-FE approximations 
in the corresponding fractional Sobolev ``energy'' norm.
The proof is based on the weighted, analytic regularity of solutions in $[0,1]^3$
proved in \cite{FMMS25-regularity3d}, 
and on a tensor product, GLL interpolant of the solution $u$.
Consequences of main result are (root-)exponential convergence rates for 
a) quantized, tensor-structured approximations as developed in \cite{CMarMRChS22},
b) neural-network based approximation as developed in \cite{MOPS23}.
Extension of the present convergence rate bound
to general, polyhedral domains $\Omega \subset \R^3$ 
shall be developed in \cite{FMMSfuture}.
\begin{acknowledgement}
Research of JMM supported by the Austrian Science Fund (FWF) project F 65.
\end{acknowledgement}
\vspace{-4mm}

\bibliographystyle{plain}
\bibliography{bibliography}

\end{document}